\documentclass{article}

\usepackage{amsmath, amssymb, amsthm}
\usepackage[margin={2cm, 2cm, 2cm, 2cm}]{geometry}
\usepackage{hyperref}
\usepackage{cite}

\usepackage{enumerate}

\numberwithin{equation}{section}

\usepackage{verbatim}
\usepackage{multirow}

\newtheorem{theorem}{Theorem}[section]
\newtheorem{lemma}[theorem]{Lemma}
\newtheorem{corollary}[theorem]{Corollary}

\newtheorem{proposition}[theorem]{Proposition}
\newtheorem{definition}[theorem]{Definition}
\newtheorem{remark}[theorem]{Remark}

\renewcommand{\subsectionmark}[1]{\markright{~}}

\newtheorem{example}[theorem]{Example}

\newcommand{\LL}{{\mathcal L}}

\newcommand{\PP}{{\mathcal P}}

\newcommand{\join}{{\vee}}
\newcommand{\meet}{{\wedge}}
\newcommand{\ad}{{a.d.}}
\newcommand{\madf}{{m.a.d.}}

\begin{document}

\title{Topologies on $X$ as points within $2^{\mathcal{P}(X)}$}

\author{
Jorge L. Bruno\\
 National University of Ireland, Galway \\ \texttt{brujo.email@gmail.com}
\and
Aisling E. McCluskey  \\ National University of Ireland, Galway \\ \texttt{aisling.mccluskey@NUIgalway.ie}
}



\maketitle

\begin{abstract}
A topology on a nonempty set $X$ specifies a natural subset of $\mathcal{P}(X)$. By identifying $\mathcal{P}(\mathcal{P}(X))$ with the totally disconnected compact Hausdorff space $2^{\mathcal{P}(X)}$, the lattice $Top(X)$ of all topologies on $X$ is a natural subspace therein. We investigate topological properties of $Top(X)$ and give sufficient model-theoretic conditions for a general subspace of $2^{\mathcal{P}(X)}$  to be compact.
\end{abstract}
\section{Introduction} Families of subsets of a nonempty set $X$ occur naturally as elements of the product space $2^{\mathcal{P}(X)}$ (where $2$ carries the discrete topology)  when one identifies $\mathcal{P}(\mathcal{P}(X))$ with $2^{\mathcal{P}(X)}$. Thus for example $Ult(X)$, $Fil(X)$, $Top(X)$ and $Lat(X)$, denoting the families of all ultrafilters, filters and topologies on $X$ and sublattices of $\mathcal{P}(X)$ respectively, occur as natural subspaces of $2^{\mathcal{P}(X)}$. Our investigation focuses principally on the topological nature of the subspace $Top(X)$ while also establishing sufficient model-theoretic conditions for a general subspace of $2^{\mathcal{P}(X)}$ to be compact. \\[5mm]
$Top(X)$ has been an object of study since its introduction by Birkhoff in \cite{MR1366862bis}. It is a complete atomic and complemented lattice under inclusion and it has undergone exploration mainly as an order-theoretic structure (\cite{MR0388306}, \cite{MR0305322}). More recent work (\cite{MR1617095}, \cite{MR1606394}, \cite{MR1924043} and \cite{MR2051095}) has concerned the order-theoretic nature of intervals of the form $[\sigma, \rho] = \{\tau \in Top(X): \sigma \leq \tau \leq \rho \}$. Our interest is motivated by the location of $Top(X)$ within the product space $2^{\mathcal{P}(X)}$ and of the subsequent topological structure it assumes from this location, about which little appears to be known. For Tychonoff $X$, the recognition of $Ult(X)$ when regarded as a subspace of $2^{\mathcal{P}(X)}$ as a homeomorphic copy of $\beta(X)$ suggests potential for this study.
\\

\section{Preliminaries and Examples}

We shall assume that $X$ is an infinite set throughout. We begin by recalling the subbasic open sets of the product space $2^{\mathcal{P}(X)}$:

\begin{definition}
Given $A \subseteq X$, we define $A^+ :=   \{\mathcal{F} \in 2^{\mathcal{P}(X)}: A \in \mathcal{F}\}$ and $ A^- :=  \{\mathcal{F} \in 2^{\mathcal{P}(X)}: A \not \in \mathcal{F}\}$ and observe that these are the subbasic open subsets of the product space $2^{\mathcal{P}(X)}$.
\end{definition}

\begin{remark}
Note that the subbasic open sets are also closed since $A^+ = 2^{\mathcal{P}(X)} \setminus A^-$. Adopting standard order-theoretic notation, we observe that $A^+ = \{A\}^{\uparrow}$ where $\theta^{\uparrow} = \{\phi \in \mathbb{P}: \theta \leq \phi \} = [\theta, \infty)$ for a partially ordered set (poset) $(\mathbb{P}, \leq)$. Given $\phi$, $\theta$ and $\psi$ in the poset $2^{\mathcal{P}(X)}$, we note the more general result that $\phi^{\downarrow} = \{\psi \mid \psi \subseteq \phi\}$ and $\phi^{\uparrow} = \{\psi \mid \phi \subseteq \psi\}$ are each closed sets. Simply observe, for example,  that for any $\theta \in 2^{\mathcal{P}(X)} \smallsetminus \phi^{\downarrow}$, there is some $B \in \theta$ such that $B \not \in \phi$. Then $B^+$ is an open neighbourhood of $\theta$ that is disjoint from $\phi^{\downarrow}$, whence $\phi^{\downarrow}$ is closed.\\
Given elements $\phi$, $\theta$ in an arbitrary poset $\mathbb{P}$, $\phi^{\uparrow}$ and $\theta^{\downarrow}$ define the subbasic closed sets for the \emph{interval topology} on  $\mathbb{P}$. Thus $2^{\mathcal{P}(X)}$'s product topology contains the associated interval topology. Note further that the product space $2^{\mathcal{P}(X)}$ is both minimal Hausdorff and maximal compact.  Thus a lattice that can be realized as a closed subset of $2^{\PP(X)}$ has a Hausdorff interval topology precisely when the lattice-morphism is also a homeomorphism.
\end{remark}

\subsection*{Examples in $Top(X)$}
\begin{example} The set of $T_1$ topologies on $X$ is closed in $Top(X)$.

\end{example}
\begin{proof} Any $T_1$ topology on $X$ must contain the cofinite topology $\mathcal{C} = \{A \subseteq X: A \mbox{ is cofinite}\} = \{ A_i : i \in I \}$.
Notice that $\mathcal{C}^{\uparrow}  =\displaystyle \cap_i A_i^+$. Then $\mathcal{C}^{\uparrow}
\cap
Top(X)$ is closed in $Top(X)$ and it is precisely the set of $T_1$ topologies on $X$.

\end{proof}

Recall that an expansive (contractive) topological property, $Q$, is one for which the collection of all topologies on a fixed set $X$ with property $Q$ is upwards (downwards) closed. As an easy consequence of the above remark, we have that if $Q$ is an expansive (contractive) topological property for which the set $Top_Q(X)$ of all $Q$ topologies on $X$ has finitely many minimal (maximal) elements, then $Top_Q(X)$ must be closed in $Top(X)$. Moreover, if $\mathbb Q = \{Q_i: i \in I\}$ is a collection of expansive (contractive) topological properties for which $Top_{Q_i}(X)$ has finitely many minimal (maximal) elements for each $i$, then $Top_{\mathbb{Q}}(X) = \{\sigma \in Top(X) \mid \sigma$ is $Q_i,$ $\forall i \in I\}$ is closed in $Top(X)$.

\begin{example} Let $f: X \rightarrow X$ be a function and $\sigma$ any topology on $X$. Then
\begin{eqnarray*}
\sigma_f & = & \{ \rho \in Top(X) \mid f: (X, \sigma) \rightarrow (X, \rho) \mbox{ is continuous} \} \mbox{ and}\\
\sigma^f & = & \{ \rho \in Top(X) \mid f: (X, \rho) \rightarrow (X, \sigma) \mbox{ is continuous} \}
\end{eqnarray*}
are closed in $Top(X)$.
\end{example}
\begin{proof} Let $\tau = \{X\} \cup \{A \mid f^{-1}(A) \in \sigma\}$. In a similar spirit to the quotient topology, this topology is the largest topology which makes $f$ continuous and thus we have $\sigma_f = \tau^{\downarrow}\cap
Top(X)$. As shown above, this set is closed in $Top(X)$. The proof is similar for $\sigma^f$.
\end{proof}
Note that if $f: X \rightarrow X$ is injective and if $\sigma$ is the cofinite topology $\mathcal{C}$, then $\sigma^f$ is precisely the set of all $T_1$ topologies on $X$.
\begin{example}Let $f: X \rightarrow X$ be a function. Then
\begin{eqnarray*}
 Cns(f)& =& \{ \rho \in Top(X) \, | \, f: (X, \rho) \rightarrow (X, \rho)\mbox{ is continuous} \},
  \\
 Open(f)& = &\{ \rho \in Top(X)\, | \, f: (X, \rho) \rightarrow (X, \rho)\mbox{ is open} \}, \mbox{ and}\\
  Closed(f)& = & \{ \rho \in Top(X)\, | \, f: (X, \rho) \rightarrow (X, \rho)\mbox{ is closed} \}
\end{eqnarray*}
 are closed in $Top(X)$.
\end{example}
\begin{proof} Observe that if $\rho \in Top(X)\smallsetminus Cns(f)$ we must have $A \in \rho$ so that $f^{-1}(A) \not \in \rho$. Then $A^+ \cap [f^{-1}(A)]^- \cap Top(X)$ is open in $Top(X)$, contains $\rho$  and  is disjoint from $Cns(f)$, whence $Cns(f)$ is closed in $Top(X)$. The proof is similar for $Open(f)$ and $Closed(f)$.
\end{proof}

In view of the above, if $f$ is a bijection then $Hom(f) = \{\sigma \in Top(X) \mid f: (X, \sigma) \rightarrow (X, \sigma) \mbox{ a homeomorphism}\}$ is also closed in $Top(X)$ (since $Hom(f) = Cns(f) \cap Open(f)$). Note also that for any function $f: X \rightarrow X$, $Cns(f)$ is a complete sublattice of $Top(X)$. \\[5mm]

\subsection*{Examples of function spaces}
Let $Y$ be an infinite set with $A, B \subseteq Y$, and denote by $\mathbb{F}^A_B$ the collection of all partial functions from $A$ to $B$; that is,  $\mathbb{F}^A_B = \{f: C \rightarrow B \mid C \subseteq A\}$. Then $[\mathbb{F}^A_B]^{<\omega} = \{f \in \mathbb{F}^A_B| |f| \in \omega\}$ is the set of all finite such partial functions. Clearly $[\mathbb{F}^A_B]^{<\omega} = \mathbb{F}^A_B$ when $A$ is finite. Consider next an example where $X$ is a special type of set derived from $Y$:

\begin{example} Let $X = Y \cup \PP(Y)$ where $Y$ is infinite and let $A, B\subseteq Y$. Then

\begin{enumerate}[(i)]
\item $\mathbb{F}^A_B$ is a compact subspace of $2^{\PP(X)}$,
\item $[\mathbb{F}^A_B]^{<\omega}$ is a proper dense subset of $\mathbb{F}^A_B$ if and only if $A$ is infinite, and
\item $[\mathbb{F}^A_B]^{<\omega}$  is not locally compact if and only if $A$ is infinite.
\end{enumerate}
\end{example}

\begin{proof}
\begin{enumerate}[(i)]
\item Note first that $\langle a,b\rangle \in \PP(X)$ for $a,b \in Y$ (that is, $\{\{a\}, \{a,b\}\} \in \PP(X))$. Thus, $\mathbb{F}^A_B \subset 2^{\PP(X)}$. If a point $p$ in $2^{\PP(X)}$ is not a function from a subset of $A$ to $B$ then either

\begin{enumerate}
\item[(1)] $p \not \subseteq A \times B$ so that $p$ contains an element $C$ where $C$ is not an element of $A \times B$. Then $C^+ \cap \mathbb{F}^A_B = \emptyset$, or
\item[(2)] $p \subseteq A \times B$ and $p$ contains at least two sets are of the form:
$\langle a,b\rangle$ and $\langle a,b'\rangle$ where $b \not = b'$. In this case, $\langle a,b\rangle^+ \cap \langle a,b'\rangle^+$ is an open neighbourhood of $p$ disjoint from
$\mathbb{F}^A_B$.
\end{enumerate}
Thus $\mathbb{F}^A_B$ is closed in $2^{\PP(X)}$ and thus compact.

\item It is easy to see that any basic open set about any $f \in \mathbb{F}^A_B$ defines a finite partial function from $A$ to $B$.

\item[(iii)]Observe from (i) and (ii) that $\mathbb{F}^A_B$ is a Tychonoff compactification of $[\mathbb{F}^A_B]^{<\omega}$ and recall that a Tychonoff space is locally compact if and only if its remainder is closed in each of its compactifications.  Let $p \in  [\mathbb{F}^A_B]^{<\omega}$ and let $ p \in  \bigcap C_i^+ \cap \bigcap D_j^- = U$. Each $C_i$ is of the
form $\langle a,b \rangle$ since $p$ is a function, while in the worst case scenario each $D_j$ has also the form $\langle a^*,b^*\rangle$. Thus any function in $U$ is compelled to include finitely many ordered pairs (as specified by each $C_i$) and at worst to exclude finitely many certain other ordered pairs within $A \times B$ (as potentially specified by $D_j$). Since $A$ is infinite, infinitely many infinite extensions of $p$ must exist in $U$. Thus $[\mathbb{F}^A_B]^{<\omega}$ has empty interior in $\mathbb{F}^A_B$. In particular, $\mathbb{F}^A_B \smallsetminus [\mathbb{F}^A_B]^{<\omega}$ is not closed and so $[\mathbb{F}^A_B]^{<\omega}$ is not locally compact.
\end{enumerate}
\end{proof}

Notice that item $(1)$ from the previous example shows that the set of all subsets from $A \times B$ is also compact.
We can define in a similar way the set $\mathbb{I}^A_B$ of injective partial functions from $A$ to $B$.

\begin{example} $\mathbb{I}^A_B$ is compact in $2^{\PP(X)}$ where $X = Y \cup \PP(Y)$ with $A,B \subseteq Y$ .
\end{example}

\begin{proof} The above example illustrates how to separate all functions in $\mathbb{F}^A_B$  from other objects in $2^{\PP(X)}$. Thus we must only focus on separating all injective functions from all other functions in $\mathbb{F}^A_B$. Let $f: C \rightarrow B$ be a function in $\mathbb{F}^A_B$ which is not injective. Then there exists distinct elements $a$, $a'$ in  $C$ so that $f(a) = f(a')$  from which it follows that $\langle a, f(a) \rangle^+ \cap \langle a' , f(a) \rangle^+ \cap \mathbb{I}^A_B = \emptyset$.
\end{proof}

Let $\mathbb{O}^A_B = \{ f \in \mathbb{F}^A_B \mid f$ is onto$\}$. The following example assumes that $|A| \geq |B|$ and the preceding notation. Of course, when $A$ is finite, then $\mathbb{O}^A_B$ is finite and thus compact.

\begin{example} $\mathbb{O}^A_B$ is a dense proper subset of $\mathbb{F}^A_B$ precisely when $A$ is infinite.
\end{example}
\begin{proof} Let $f \in \mathbb{F}^A_B$ and let $\bigcap C_i^+ \cap \bigcap D_j^-$ be a basic open neighbourhood of $f$. As described in the proof of Example 2.6 (iii)  since $A$ is infinite,  it is straightforward to find an onto function from $A$ to $B$ in this neighbourhood.
\end{proof}

\subsection*{Almost disjoint families}

\begin{definition} Let $\kappa$ be an infinite cardinal. If $x,y \subset \kappa$ and $|x \cap y| < \kappa$, then $x$ and $y$ are said to be \textbf{almost disjoint} ($\ad$). An a.d. \textbf{family}  is a collection $\mathcal{A}\subseteq \PP(\kappa)$ so that for all $x \in \mathcal{A}$, $|x| = \kappa$ and any two distinct elements of $\mathcal{A}$ are a.d. Such a family is \textbf{maximal} ($\madf$) whenever it is not contained in any other $\ad$ family.
\end{definition}

The following lemma shows that given a cardinal $\kappa$ the collection of all $\ad$ families of $\kappa$ is a compact subset of $2^{\PP(\kappa)}$.

\begin{lemma} \label{lem:ad} Let $\kappa$ be any cardinal. For any cardinal $\lambda \leq \kappa$ the set $\mathcal{A}_{\lambda} = \{\mathcal{A} \in 2^{\PP(\kappa)} \mid \forall x, y \in \mathcal{A}$, $|x| = |y| = \kappa$ and $|x \cap y| < \lambda\}$ is compact in $2^{\PP(\kappa)}$. In particular, the collection of all almost disjoint families (i.e. $\mathcal{A}_{\kappa}$) is compact.
\end{lemma}

\begin{proof} Note that if $\alpha \in p \in 2^{\PP(\kappa)}$ with $|\alpha| < \kappa$, then $\alpha^+ \cap \mathcal{A}_{\lambda} = \emptyset$. Suppose that $p \in 2^{\PP(\kappa)} \smallsetminus \mathcal{A}_{\lambda}$ and that for all $\alpha \in p$, $|\alpha| = \kappa$. Then $p$ must contain at least two subsets of $\kappa$, $\alpha$ and $\beta$ say, so that $|\alpha \cap \beta| \geq \lambda$. In turn, $\alpha^+ \cap \beta^+ \cap \mathcal{A}_{\lambda} = \emptyset$.

\end{proof}

In light of the above, for a given cardinal $\kappa$, in ZFC we must have that the cardinality of any maximal chain in $\mathcal{A}_{\kappa}$ is at least $\kappa^+$. To see this, take any chain $C$ with cardinality $\kappa$ and notice that $\bigcup C \in \overline{C} \subset \mathcal{A}_{\kappa}$. If $|\bigcup C| = \kappa^+$ then we can always augment $C$ to a chain of size. If not, then a simple diagonal argument shows that $\bigcup C$ is not a $\madf$ family and in turn $C$ is not maximal $\kappa^+$ (\cite{MR756630} pg. 48).

\section{Main results}

We begin with some preliminary definitions:
\begin{definition} For $X$ an infinite set, $Lat(X)$ denotes the set of all sublattices of $\PP(X)$, $Lat(X,\join)$ is the set of all join complete sublattices of $\PP(X)$ ($Lat(X,\meet)$ is defined dually) and $LatB(X)$ is the set of all sublattices of $\PP(X)$ that contain $\emptyset$ and $X$.
\end{definition}

\begin{lemma} $Lat(X)$ is closed in $2^{\PP(X)}$ and each of $Lat(X, \join)$ and $Lat(X, \meet)$ is  dense in $Lat(X)$.

\end{lemma}

\begin{proof} It is routine to show that an element $\theta$ of $2^{\PP(X)}$ that is not a sublattice of $\PP(X)$ has a basic open neighbourhood that is disjoint from $Lat(X)$. The failure of $\theta$ to be a sublattice provides sets, according to how the failure is manifested, with which to assemble an appropriate open neighbourhood. For instance, if $\theta$ is not closed under binary joins then we can find $A, B \in \theta$ for which $A \cup B \not \in \theta$. In turn, $A^+ \cap B^+ \cap (A \cup B)^- \cap Lat(X) = \emptyset$.
 Next, let $L$ be a sublattice of $\PP(X)$ and let $\bigcap A_i^+ \cap \bigcap B_j^-$ be a basic open neighbourhood of  $L$. Since $L$ is a lattice, the join of any family of $A_i$s can not yield any $B_j$. Thus we can generate from the $A_i$s (by closing under intersection and union) a finite and therefore complete lattice .

\end{proof}

It is important to notice that any join complete lattice in $LatB(X)$ is a topology on $X$. Similarly, any infinite subblatice in $LatB(X)$ that is not join complete fails to be a topology on $X$. Consequently, $Top(X) \subset LatB(X)$. 

\begin{proposition} \label{prop:compact} $LatB(X)$ is a (Hausdorff) compactification of $Top(X)$.

\end{proposition}

\begin{proof} To see that $LatB(X)$ is closed, and therefore compact,  in $2^{\PP(X)}$, it suffices by Lemma 3.2 to consider those sublattices of  $\PP(X)$ that omit either $X$ or $\emptyset$. We can then use $X^-$ or $\emptyset^-$, as appropriate, as a neighbourhood of the given sublattice that is disjoint from $LatB(X)$. Finally, if $L \in LatB(X)$ with $\bigcap A_i^+ \cap \bigcap B_j^-$ an open neighbourhood of $L$, then no $B_j$ can be the meet or join of any collection of $A_i$s. Thus we can define a (finite) topology with subbase $\{A_i\}$ which clearly belongs to the given neighbourhood. That is, $Top(X)$ is dense in $LatB(X)$.

\end{proof}

The previous theorem indicates that the concept of generalized topologies on a fixed set $X$ is that of sublattices of $\PP(X)$ containing $X$ and $\emptyset$. That is, limits of topologies are just elements from $LatB(X)$. 

\begin{corollary}
$Top(X)$ is not closed, and therefore not compact, in $2^{\PP(X)}$.
\end{corollary}
\begin{proof}
$LatB(X)$ properly contains $Top(X)$ as a dense subset. Since $LatB(X)$ is closed in
$2^{\PP(X)}$, then $\overline{Top(X)} = LatB(X)$. Thus $Top(X)$ can be neither closed nor compact in $2^{\PP(X)}$.
\end{proof}

\begin{lemma} \label{lem:empty} $Top(X)$ has empty interior in $LatB(X)$.

\end{lemma}
\begin{proof} Take any topology $\tau$ on $X$ and a basic open neighbourhood $\bigcap A_i^+\cap \bigcap B_j^-$ of $\tau$.
We seek an infinite subset $S$ of $X$ and a partition $\mathcal{P} = \{P_k\}_{k \in \kappa}$ of $S$ into infinitely many subsets, whose finite joins and meets cannot be any $B_j$. $\mathcal{P}$ together with the $A_i$s will generate an element of $LatB(X)$  that is in the above-mentioned neighbourhood and that is not join complete (and thus not a topology). There are two possibilities:  either the complement of $\bigcup A_i \cup \bigcup B_j$ is finite or it is not.
In the latter case, just take $S$ to be the (infinite) complement, in which case any partition of $S$ into infinitely many subsets will suffice.
 For the former case,  rename all infinite $A_i$s and $B_j$s as $C_k$s. Let $\mathcal{C}_1 = C_1$,
\[
\mathcal{C}_2 = \left\{
\begin{array}{l l}
  \mathcal{C}_1 \cap C_2,& \quad \text{if } |\mathcal{C}_1 \cap C_2| \geq \aleph_0\\
  \\
  \mathcal{C}_1, & \quad \text{otherwise.}\\
\end{array} \right.
\]\\

In general,

\[
\mathcal{C}_k = \left\{
\begin{array}{l l}
  \mathcal{C}_{k-1} \cap C_k,& \quad \text{if } |\mathcal{C}_{k-1} \cap C_k| \geq \aleph_0\\
  \\
  \mathcal{C}_{k-1}, & \quad \text{otherwise.}\\
\end{array} \right.
\]\\
The process terminates with an infinite set $\mathcal{C}_n$. Let $S$ be that subset of $\mathcal{C}_n$ obtained by removing any elements that occur in intersections of finite cardinality with  $A_i$s and  $B_j$s (that is, those that did not partake in constructing $\mathcal{C}_n$). That is,

$$
S = \mathcal{C}_n \smallsetminus \left[\bigcup_{|\mathcal{C}_n \cap A_i| \in \omega} (\mathcal{C}_n \cap A_i) \cup \bigcup_{|\mathcal{C}_n \cap B_j| \in \omega} (\mathcal{C}_n \cap B_j) \right].
$$

Then any partition $\mathcal{P} = \{P_k\}_{k \in \kappa}$ of $S$ into infinitely many subsets together with the $A_i$s will generate the required family $L$, say,  in $LatB(X)$. Indeed, note that by design $A_i \cap S = \emptyset$ or $A_i \cap S = S$, $\forall i$ (the same is true for all $B_j$). It remains to show that no $B_j$ belongs to $L$. Let $k$ be any number so that $B_k \cap S = \emptyset$. Then for $B_k \in L$ it must be that $B_k$ is generated solely by the use of $A_i$s (contradiction). Otherwise, $B_k \cap S = S$. To this end, notice that no finite collection of elements from $\mathcal{P}$ can yield $S$. Consequently, we find again that $B_k$ must then be generated by taking a finite union of $A_i$s (contradiction).

\end{proof}

\begin{corollary} $Top(X)$ is not locally compact.
\end{corollary}

\begin{proof}  By Lemma~\ref{lem:empty}, $LatB(X)\smallsetminus Top(X)$ is not closed and the result follows.

\end{proof}

\begin{corollary}
$Top(X)$ is dense and co-dense in $LatB(X)$.
\end{corollary}
\begin{proof} The proof is immediate from Proposition~\ref{prop:compact} and Lemma~\ref{lem:empty}.
\end{proof}
\newpage
\subsection*{Some Model Theory}
Let $\LL$ be the first order language of Boolean algebras ( \cite{MR1707268}) and $P$ a predicate unary symbol. Expand $\LL$
by adding $P$ and denote the expanded language by $\LL(P)$ (\cite{MR1386188}).

\begin{definition} Let $\Phi$ be a set of sentences of $\LL(P)$. We say that $\Phi$ is a definition for a set $\mathcal{F} \subset 2^{\PP(X)}$ provided that $\mathcal{F}$ is the collection of all $F \subset \PP(X)$ so that $(\PP(X),F) \models \Phi$. Similarly, a set of sentences $\{\Phi_i\}$ defines a family $\mathcal{F}$ provided that for each $F \in \mathcal{F}$ we have $(\PP(X), F) \models \Phi_i$ for all $i$.
\end{definition}

\begin{theorem} \label{thm:model1} A subspace of $2^{\PP(X)}$ is compact if it has a definition expressible as a universal sentence of $\LL(P)$.
\end{theorem}

\begin{proof} Let $\Phi(x_1, x_2, \ldots, x_n, P)$ be a well formed formula. We can assume that $\Phi$ is in conjunctive normal form; that is, $\Phi$ is a conjunction of atomic formulas and negation of atomic formulas. Thus  $\Phi(x_1, x_2, \ldots, x_n, P) = \alpha_1(x_1, x_2, \ldots, x_n, P) \wedge \ldots \wedge \alpha_k(x_1, x_2, \ldots, x_n, P)$ where each $\alpha_i$ is atomic or a negation of an atomic formula. Note that there is no loss no generality here since universal quantifiers distribute over conjunctions. We need only prove the theorem for an atomic formula since intersections of closed sets are closed. Suppose then that $\Phi(x_1, x_2, \ldots, x_n, P)$ is atomic.
Define $\Phi^*(P) = \forall$ $x_1, x_2, \ldots, x_n$ $\Phi(x_1, x_2, \ldots, x_n, P)$ and let $\mathcal{F} = \{ S \subseteq \PP(X)$ $|$ $ (\PP(X), S) \models
\Phi^*(P)\}$. If $S' \not \in \mathcal{F}$ then $(\PP(X), S') \not \models \Phi^*(P)$, thus $\exists$ $A_1, A_2, \dots A_n \in S'$ so that $(\PP(X), S')
[A_1/x_1, A_2/x_2, \ldots, A_n/x_n]$ does not satisfy $\Phi(x_1, x_2, \ldots, x_n, P)$. Hence, no element of $\mathcal{F}$ can contain all of the
$A_i$s. It follows that $A_1^+ \cap A_2^+ \cap \ldots \cap A_n^+ \cap \mathcal{F} = \emptyset$.\\

\end{proof}

\begin{corollary} A subspace of $2^{\PP(X)}$ is compact if it has a definition expressible as a collection of universal sentences of $\LL(P)$.
\end{corollary}

\begin{proof} Let $\{\Phi_i\}$ be a collection of universal sentences in $\LL(P)$. Theorem~\ref{thm:model1} proves that each sentence defines a compact subspace $\mathcal{F}_i \subset 2^{\PP(X)}$. Lastly, $\{\Phi_i\}$ is then a definition for $\bigcap_i \mathcal{F}_i$, which is closed and thus compact.
\end{proof}

\begin{corollary} $Top(X)$ is not first-order definable via universal sentences.
\end{corollary}

\begin{remark} Clearly, not all compact sets can be defined in terms of universal sentences for there are not enough sentences in our countable language to capture them all. What is not as obvious is that there are instances where the use of a unary predicate symbol is essential. Take for instance $\mathcal{A}_{\kappa}$ as defined in Lemma~\ref{lem:ad}. Add to $\mathcal{L}$ a unary predicate symbol $P$ with the following intended interpretation: $\forall x$ $P(x) \iff |x| = \kappa$. Then we can easily define $\mathcal{A}_{\kappa}$. However, the same is not true with the use of $\mathcal{L}$ alone.

Lastly, we can expand $\mathcal{L}$ to an uncountable language by adding for all $Z \subset X$ a unary predicate symbol $P_Z$. Given $Y \subset X$, we want $P_Z(Y)$ to be true if and only if $Z = Y$. It is then possible to define all subbasic closed sets of the form $A^-$ for if we let $\Phi_A= \forall x\neg P_A(x)$, then $A^-$ is the unique subset of $2^{\PP(X)}$ for which $(\PP(X), A^-) \models \Phi_A$. Similarly, $A^+$ is the unique set so that $(\PP(X), A^+) \models \neg \Phi_A$. In turn, any basic closed set satisfies a sentence expressed as a finite disjunction of the aforementioned `subbasic' sentences. Finally, a closed subset in $2^{\PP(X)}$ satisfies an arbitrary collection of `basic' sentences as described above. 
\end{remark}

\section{Conclusion}
This paper offers results from a preliminary investigation into the topological nature of $Top(X)$ as a subspace of $2^{\PP(X)}$. The outcomes thus far suggest that $Top(X)$'s topological character may well be difficult to gauge. Knowing whether $Top(X)$ is a $G_\delta$ or an $F_\sigma$, for example,  will shed further light on its topological complexity and crystallize the extent to which it may yield to further analysis. In this paper, we have exhibited examples of sets that are closed in $Top(X)$ but not  in $2^{\PP(X)}$. It would be nice to find (infinite) subsets  of $Top(X)$ that are closed in $2^{\PP(X)}$, thus giving a source of infinite compact subsets of $Top(X)$. For example, is it possible to find a copy of $\beta \mathbb{N}$ or a Cantor set in $Top(X)$? Answers to these questions would provide further interesting aspects of $Top(X)$.

\section{Acknowledgement}
It is a pleasure to record our indebtedness to Paul Bankston, Marquette University, Milwaukee for valuable conversations and insights in pursuit of these results.

\bibliographystyle{amsplain}
\bibliography{mybib2}

\providecommand{\bysame}{\leavevmode\hbox to3em{\hrulefill}\thinspace}
\providecommand{\MR}{\relax\ifhmode\unskip\space\fi MR }
\providecommand{\MRhref}[2]{%
  \href{http://www.ams.org/mathscinet-getitem?mr=#1}{#2}
}
\providecommand{\href}[2]{#2}
\begin{thebibliography}{10}

\bibitem{MR1366862bis}
Garrett Birkhoff, \emph{On the combination of topologies}, Fund. Math.
  \textbf{26} (1936), no.~1, 156--166.

\bibitem{MR1924043}
C.~Good, D.~W. McIntyre, and W.~S. Watson, \emph{Measurable cardinals and
  finite intervals between regular topologies}, Topology Appl. \textbf{123}
  (2002), no.~3, 429--441. \MR{1924043 (2003h:54001)}

\bibitem{MR1617095}
R.~W. Knight, P.~Gartside, and D.~W. McIntyre, \emph{All finite distributive
  lattices occur as intervals between {H}ausdorff topologies}, Proceedings of
  the {E}ighth {P}rague {T}opological {S}ymposium (1996), Topol. Atlas, North
  Bay, ON, 1997, pp.~167--172 (electronic). \MR{1617095}

\bibitem{MR756630}
Kenneth Kunen, \emph{Set theory}, Studies in Logic and the Foundations of
  Mathematics, vol. 102, North-Holland Publishing Co., Amsterdam, 1983, An
  introduction to independence proofs, Reprint of the 1980 original. \MR{756630
  (85e:03003)}

\bibitem{MR0388306}
Roland~E. Larson and Susan~J. Andima, \emph{The lattice of topologies: a
  survey}, Rocky Mountain J. Math. \textbf{5} (1975), 177--198. \MR{0388306 (52
  \#9143)}

\bibitem{MR1386188}
Mar{\'{\i}}a Manzano, \emph{Extensions of first order logic}, Cambridge Tracts
  in Theoretical Computer Science, vol.~19, Cambridge University Press,
  Cambridge, 1996. \MR{1386188 (97i:03002)}

\bibitem{MR1707268}
\bysame, \emph{Model theory}, Oxford Logic Guides, vol.~37, The Clarendon Press
  Oxford University Press, New York, 1999, With a preface by Jes{\'u}s
  Moster{\'{\i}}n, Translated from the 1990 Spanish edition by Ruy J. G. B. de
  Queiroz, Oxford Science Publications. \MR{1707268 (2000f:03001)}

\bibitem{MR1606394}
D.~W. McIntyre and W.~S. Watson, \emph{Basic intervals in the partial order of
  metrizable topologies}, Topology Appl. \textbf{83} (1998), no.~3, 213--230.
  \MR{1606394 (99c:54006)}

\bibitem{MR2051095}
\bysame, \emph{Finite intervals in the partial orders of zero-dimensional,
  {T}ychonoff and regular topologies}, Topology Appl. \textbf{139} (2004),
  no.~1-3, 23--36. \MR{2051095 (2005b:54006)}

\bibitem{MR0305322}
Richard Valent and Roland~E. Larson, \emph{Basic intervals in the lattice of
  topologies}, Duke Math. J. \textbf{39} (1972), 401--411. \MR{0305322 (46
  \#4452)}

\end{thebibliography}

\end{document}